\documentclass[a4paper,11pt]{amsart}

\usepackage[activeacute,english]{babel}
\usepackage[latin1]{inputenc}
\usepackage{amsmath,amssymb,amsfonts,amsthm,amscd}
\usepackage[mathscr]{eucal}
\usepackage{latexsym}
\usepackage{graphicx}
\usepackage{multicol}
\usepackage[active]{srcltx}
\usepackage[all]{xy}
\usepackage[T1]{fontenc}
\usepackage[sc]{mathpazo}

\newtheoremstyle{teorema}{18pt}{15pt}{\parskip=0pt\itshape}
{}{}{}{ }{\textbf{\thmname{#1} \thmnumber{#2}\thmnote{ (#3)}.}}
\newtheoremstyle{plano}{18pt}{15pt}{\parskip=0pt}
{}{}{}{ }{\textbf{\thmname{#1} \thmnumber{#2}.} \thmnote{(\textit{#3})}}

\theoremstyle{teorema}
\newtheorem{teor}{Theorem}[section]
\newtheorem{lema}[teor]{Lemma}
\newtheorem{coro}[teor]{Corollary}
\newtheorem{prop}[teor]{Proposition}

\theoremstyle{plano}
\newtheorem{obse}[teor]{Remark}

\newcommand{\Z}{\mathbb{Z}}

\newcommand{\R}{\mathbb{R}}

\renewcommand{\H}{\mathbb{H}}

\newcommand{\sm}{\smallsetminus}

\newcommand{\Sp}{\mathbb{S}}

\newcommand{\M}{\mathbb{M}}

\newcommand{\wt}{\widetilde}

\newcommand{\p}{\partial}

\DeclareMathOperator{\sen}{sen}
\DeclareMathOperator{\arctanh}{arctanh}
\DeclareMathOperator{\arccosh}{arccosh}

\newcommand{\dt}{\,\mathrm{d}}
\newcommand{\dist}{\mathop{\rm dist}\nolimits}
\newcommand{\Ric}{\mathop{\rm Ric}\nolimits}

\newcommand{\arcsinh}{\mathop{\rm arcsinh}\nolimits}

\newcommand{\Long}{\mathop{\rm Long}\nolimits}

\author{José M. Manzano}
\title{Estimates for constant mean curvature graphs in $M\times\R$}
\thanks{Research partially supported by a Spanish MEC-FEDER Grant no. MTM2007-61775 and a Regional J. Andaluc\'\i a Grant no. P06-FQM-01642.}
\address{José Miguel Manzano, Dpto. Geometría y Topología, Universidad de Granada. Email address: {\tt jmmanzano@ugr.es}}

\begin{document}

\begin{abstract}
We will discuss some sharp estimates for a constant mean curvature graph $\Sigma$ in a Riemannian 3-manifold $M\times\R$ whose boundary $\partial\Sigma$ is contained in a slice $M\times\{t_0\}$. We will start by giving sharp lower bounds for the geodesic curvature of the boundary and improve these bounds when assuming additional restrictions on the maximum height that such a surface reaches in $M\times\R$. We will also give a bound for the distance from an interior point to the boundary in terms of the height at that point, and characterize when these bounds are attained.
\end{abstract}

\maketitle

\section{Introduction}

\noindent Constant mean curvature surfaces in several $3$-manifolds have been extensively studied in recent times. One of the most important families of such $3$-manifolds are product spaces $M\times\R$, $M$ being a Riemannian surface, which includes the homogeneous spaces $\R^3$, $\H^2\times\R$ and $\Sp^2\times\R$. It was Rosenberg in \cite{rose2} who started the study of minimal surfaces in $M\times\R$ and, since then, many papers in this setting have appeared.

We will focus on constant mean curvature $H>0$ graphs $\Sigma$ ($H$-graphs in the sequel) in $M\times\R$ whose boundary $\partial \Sigma$ lies in some slice $M\times\{t_0\}$. If we denote by $c$ the infimum of the Gaussian curvature $K_M$ of the domain $\Omega\subset M$ over which $\Sigma$ is a graph, we will assume the hypothesis $4H^2+c>0$. It is worth mentioning that the sign of $4H^2+c$ makes a qualitative difference in the geometry of $H$-surfaces in $\M^2(c)\times\R$, where $\M^2(c)$ will stand for the simply connected surface with constant Gaussian curvature $c$. For instance. $4H^2+c>0$ is the natural condition for the existence of constant mean curvature $H$ spheres in $\M^2(c)\times\R$. In fact, most results in the present paper can be understood as comparison results between the geometries of $M\times\R$ and the corresponding homogeneous case $\M^2(c)\times\R$.

Moreover, we will restrict ourselves to a capillarity problem, i.e., when $\Sigma$ has constant angle function $\nu=\nu_0$ for some $-1<\nu_0\leq0$ along its boundary. Note that the angle function of $\Sigma$ is defined as $\nu=\langle N,E_3\rangle$, where $N$ is the unit normal vector field to $\Sigma$ for which the mean curvature of $\Sigma$ is $H$, and $E_3=\partial_t$ is the vertical Killing vector field. This problem is a classical example of overdetermined situation so it is not expected that many surfaces satisfy those conditions for a given domain $\Omega\subset M$. Although no capillary graphs for $\nu_0\neq0$ are known in $M\times\R$ different from those invariant under a $1$-parameter group of isometries, it turns out that compact $H$-surfaces in $M\times\R$ lie in this case (for $\nu_0=0$) as a consequence of Alexandrov reflection principle~\cite{Alex}. They are contained in the more general case of embedded $H$-bigraphs, i.e., (not necessarily compact) connected embedded $H$-surfaces which are made up of two graphs, symmetric with respect to some slice $M\times\{t_0\}$. 

We will now list some examples of this kind of surfaces. Ritoré \cite{rit1} and Große-Brauckmann \cite{gb1} constructed certain families of non-compact $H$-bigraphs in $\R^3$. In the more general case of $\M^2(c)\times\R$, there are rotationally invariant $H$-spheres and $H$-cylinders invariant under a $1$-parameter group of isometries around a geodesic which are $H$-bigraphs for $4H^2+c>0$ (see Section~\ref{sec_ejemplos}). In fact, the results proved in this paper can be applied to some not necessarily embedded $H$-bigraphs which are periodic with a compact fundamental piece, as in the horizontal unduloids in $\M^2(c)\times\R$ constructed by the author and Torralbo in~\cite{MT}. Finally, in a general product manifold $M\times\R$, for $H$ large enough, embedded constant mean curvature spheres $H$ in $M\times\R$ exist as solutions of the isoperimetric problem as well as certain perturbations of tubular neighborhoods around horizontal geodesics, see Mazzeo and Pacard~\cite{MazPac}. 

Most results in this paper will concern estimates for the geodesic curvature (in $M\times\{0\}$) of the boundary of an $H$-bigraph $\Sigma\subset M\times\R$ depending on the height that $\Sigma$ reaches (the height function $h\in C^\infty(\Sigma)$ is given by $h(p,t)=t$). Ros and Rosenberg proved in~\cite[Theorem 8]{ror1} that any properly embedded $H$-bigraph in $\R^3\equiv\R^2\times\R$ over a domain $\Omega\subset\R^2$ with height function such that $|h|\leq\frac{1}{H}$, satisfies that the components of $\R^2-\Omega$ are strictly convex. We generalize this result to $\M^2(c)\times\R$ and improve the estimate on the geodesic curvature when the maximum height of the surface is assumed to be small enough.

Observe that, as $H$-graphs are stable, the condition $4H^2+c>0$ makes possible to apply Theorem 2.8 in \cite{mpr19} to conclude that the distance function $d(p,\partial\Sigma)$, $p\in\Sigma$, is bounded, so the height function is also bounded (in the case $4H^2+c\leq0$, this property fails to be true as invariant examples in $\H^2\times\R$ given in \cite{onnis1} and \cite{earp1} show). Aledo, Espinar and Gálvez proved in \cite{aleg1} that if $\Sigma\subseteq M\times\R$ is an $H$-graph over a compact open domain that extends to its boundary with $h=0$ and $\nu=\nu_0$ in $\partial\Omega$, and $c=\inf\{K_M(p):p\in\Omega\}>-4H^2$, then $\Sigma$ can reach at most height
\begin{equation}\label{altura_maxima}
\alpha(c,H,\nu_0)=\begin{cases}
\frac{4H}{\sqrt{-4cH^2-c^2}}\left(\arctan\left(\frac{\sqrt{-c}}{\sqrt{c+4H^2}}\right)+\arctan\left(\frac{\nu_0\sqrt{-c}}{\sqrt{c+4H^2}}\right)\right)&\text{if } c<0,\\
\frac{1+\nu_0}H&\text{if }c=0,\\
\frac{4H}{\sqrt{4cH^2+c^2}}\left(\arctanh\left(\frac{\sqrt{c}}{\sqrt{c+4H^2}}\right)+\arctanh\left(\frac{\nu_0\sqrt{c}}{\sqrt{c+4H^2}}\right)\right)&\text{if } c>0.
\end{cases}
\end{equation}
Although they only considered the case $\nu_0=0$, their argument can be directly generalized. This bound turns out to be the best one in terms of $(c,H,\nu_0)$ in the sense that  the only such $H$-graphs in $\M^2(c)\times\R$ for which equality holds are spherical caps of rotationally invariant spheres meeting $\M^2(c)\times\{0\}$ with constant angle $\nu_0$.

Theorems~\ref{teor:acotacion_kg} and~\ref{thm:acotacion-kg-adicional} will state the following results (for $\nu_0=0$) in the case the regular domain $\Omega\subset M$ is compact (see also Remarks~\ref{rmk:acotacion_kg_general} and~\ref{rmk:acotacion-kg-adicional} for arbitrary $-1<\nu_0\leq 0$):
\begin{itemize}
\item The geodesic curvature $\kappa_g$ of $\partial\Omega$ in $M$, with respect to the outer conormal vector field, satisfies the lower bound
\[\kappa_g\geq\frac{-4H^2+c(1-\nu_0^2)}{4H\sqrt{1-\nu_0^2}},\]
and, when $M=\M^2(c)$, equality holds only for rotationally invariant spheres.
\item If we additionally suppose that $|h|\leq m\cdot\alpha(c.H,\nu_0)$ for some constant $0<m\leq\frac{1}{2}$, then the previous bound is improved to the following one:
\[\kappa_g\geq\frac{(4-8m)H^2+c(1-\nu_0^2)}{4mH\sqrt{1-\nu_0^2}}.\]
\end{itemize}
In Theorem~\ref{teor_homogeneos}, we will drop the compactness hypothesis in the second item above when we restrict to $M=\M^2(c)$ and $\nu_0=0$. In this case, equality holds if and only if $m=\frac{1}{2}$ and $\Sigma$ is an $H$-cylinder invariant under a $1$-parameter group of horizontal isometries (these examples are described in section \ref{sec_ejemplos}). Observe that $m$ represents the fraction of the maximum height that $\Sigma$ is allowed to reach; it is remarkable that the maximum height of the invariant horizontal  $H$-cylinder is exactly one half of the maximum height of the corresponding $H$-sphere in $\M^2(c)\times\R$, which makes the value $m=\frac{1}{2}$ special. Hence, we extend the results by Ros and Rosenberg in~\cite{ror1}, where the case $M=\R^2$ and $m=\frac{1}{2}$ is treated.

Finally, in Section~\ref{sec:length-estimates} we will give another application of the same techniques to obtain a sharp lower bound for the distance from a point in $\Sigma$ to $\partial\Sigma$. Let us highlight that, in this last section, no capillarity condition or height restriction is assumed.

The author would like to thank Joaquín Pérez, Magdalena Rodríguez and Francisco Torralbo for some helpful conversations.

\section{Invariant surfaces in $\H^2\times\R$ and $\Sp^2\times\R$}\label{sec_ejemplos}

\noindent In this section, we will study surfaces that are invariant by $1$-parameter groups of isometries in $\M^2(c)\times\R$ which act trivially on the vertical lines. In fact, among these, we are interested in surfaces which are $H$-bigraphs (i.e. embedded $H$-surfaces symmetric with respect to a horizontal slice), for $H>0$ and $4H^2+c>0$. Thus, these groups of isometries can be identified with $1$-parameter groups of isometries of the base $\M^2(c)$.

In $\H^2$, there exist three different types of $1$-parameter groups of isometries, namely, rotations around a point, parabolic translations (i.e. rotations about a point at infinity) and hyperbolic translations. The family of rotationally invariant $H$-surfaces in $\H^2\times\R$ was studied by Hsiang and Hsiang \cite{hshs1} and those invariant by the other two families (including screw motion) were also studied by Sa Earp \cite{earp1} but it was Onnis \cite{onnis1} who gave a full classification of all invariant $H$-surfaces in $\H^2\times\R$. The case of $\Sp^2$ is quite different, because the only $1$-parameter groups of isometries of $\Sp^2$ are the rotations around a certain point and, up to conjugation, this point can be supposed to be the north pole. Such rotationally invariant $H$-surfaces were classified by Pedrosa \cite{ped1}. 

Finally, the only $1$-parameter groups of isometries of $\R^2$ are rotations around a point and translations; the former give rise in $\R^3=\R^2\times\R$ to Euclidean spheres of radius $\frac{1}{H}$, the latter to horizontal cylinders of radius $\frac{1}{2H}$.

For the sake of completeness, we will now derive the parametrizations and formulas that we will need in each of these situations. We will begin with rotations in both $\H^2\times\R$ and $\Sp^2\times\R$ and then proceed to parabolic and hyperbolic translations in $\H^2\times\R$. Let us recall that, up to a homothety, we can suppose $c\in\{-1,0,1\}$ and, in the cases $c=1$ and $c=0$, the condition $4H^2+c>0$ is meaningless (as $H>0$) but, for $c=-1$, it implies that $H>\frac{1}{2}$.

\subsection{Rotationally invariant surfaces in $\H^2\times\R$ and $\Sp^2\times\R$}\label{sec_ejemplos_rotacionales}
To start with, let us consider the model $\H^2\times\R=\{(x,y,z,t)\in\R^4:x^2+y^2-z^2=-1,z>0\}$ endowed with the metric $dx^2+dy^2-dz^2+dt^2$. It was shown by Hsiang and Hsiang that, for any $H>\frac{1}{2}$, the only rotationally invariant $H$-bigraphs are the rotationally invariant CMC spheres. If we suppose the axis of rotation to be $\{(0,0,1)\}\times\R$, the upper half of such a sphere is parametrized by $X(r,u)=\left(\sinh r\cos u,\sinh r\sin u,\cosh r,h(r)\right)$, 
where $u\in\R$ and
\[h(r)=\frac{4H}{\sqrt{4H^2-1}}\arcsin\sqrt{\frac{1-(4H^2-1)\sinh^2\frac{r}{2}}{4H^2}},\quad r\in\left[0,2\arcsinh\frac{1}{\sqrt{4H^2-1}}\right]\]
(see figure \ref{fig:H2xR} where some examples have been depicted). 
\begin{figure}
\centering\includegraphics[height=6cm]{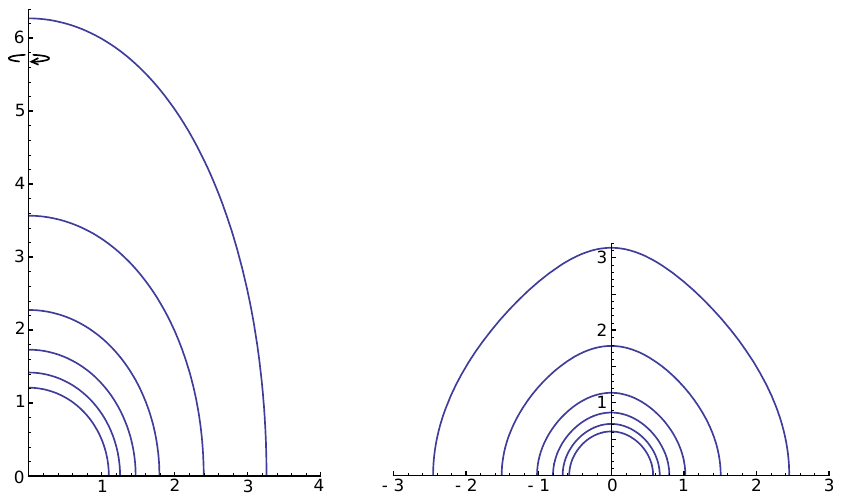}
\caption{On the left, rotationally invariant CMC spheres (the horizontal axis represents the intrinsic length in $\H^2$ and the vertical one is the real line) and, on the right,  CMC cylinders invariant under hyperbolic translations, where we see their intersection with the plane $y=1$ in the halfspace model. In both cases, the represented values of $H$ are 0.54, 0.6, 0.7, 0.8, 0.9 and 1.}\label{fig:H2xR}
\end{figure}

On the other hand, we will consider the standard model of $\Sp^2\times\R$ as a submanifold of $\R^4$, given by $\Sp^2\times\R=\{(x,y,z,t)\in\R^4:x^2+y^2+z^2=1\}$ with the induced Riemannian metric. It is well-known that every $1$-parameter group of ambient isometries consists only of rotations so, up to an isometry, they may be supposed to be rotations around the axis $\{(0,0,1,t):t\in\R\}$. Hence, the orbit space can be identified with the totally geodesic surface $\{(x,y,z,t)\in\Sp^2\times\R:x=0\}\cong\Sp^1\times\R$, and we will take the generating curve as
\[\gamma(t)=\left(0,\sin r(t),\cos r(t),h(t)\right)\]
for some functions $r,h$ defined on some interval of the real line. Pedrosa \cite{ped1} showed that the generated surface has constant mean curvature $H\in\R$ if and only if certain ODE system is satisfied. In fact, he proved that in the intervals where $r$ is invertible, we can take it as the parameter and the corresponding ODE system becomes
\begin{equation}\label{eqn:sistema_S2xR}
\left\{\begin{array}{l}
h'(r)=\cot(\sigma(r)),\\
\sigma'(r)=\frac{2H+\cot(r)\cos(\sigma(r))}{\sen(\sigma(r))},
\end{array}\right.
\end{equation}
for an auxiliary function $\sigma$. The second equation can be easily solved as it only depends on $r$ and $\sigma$ and we obtain
\[\sigma(r)=\arccos(2H(c_0+\cos r)\csc r)\]
for some $c_0\in\R$, where $r\in[a(c_0),b(c_0)]\subseteq[-\pi,\pi]$ is the maximal interval in which $\sigma$ is defined. By plugging this expression into the first equation in (\ref{eqn:sistema_S2xR}), we arrive to 
\begin{equation}\label{eqn:S2xR_h}
 h(r)=\int_{a(c_0)}^r\frac{2H(c_0+\cos s)\csc s}{\sqrt{1-4H^2(c_0+\cos s)^2\csc^2 s}}\dt s.
\end{equation}
The only two cases which lead to $H$-bigraphs are the following:
\begin{itemize}
 \item For $c_0=-1$, rotationally invariant spheres are obtained. More explicitly, 
\[h(r)=\frac{4H}{\sqrt{1+4H^2}}\arccosh\left(\frac{\sqrt{1+4H^2}}{2H}\cos\frac{r}{2}\right)\]
where $r$ lies in the interval $[-2\arctan\frac{1}{2H},2\arctan\frac{1}{2H}]$. Thus, the maximum height is attained for $r=0$ and that sphere is a bigraph over a domain whose boundary has constant geodesic curvature in $\Sp^2$ with respect to the outer conormal vector field, equal to $-H+\frac{1}{4H}$.
\item For $c_0=0$, we obtain rotationally invariant tori instead. In this case,
\[h(r)=\frac{2H}{\sqrt{1+4H^2}}\arccosh\left(\frac{\sqrt{1+4H^2}}{2H}\sin r\right)\]
where $r\in[\frac{\pi}{2}-\arctan\frac{1}{2H},\frac{\pi}{2}+\arctan\frac{1}{2H}]$. The maximum height is attained when $r=\frac{\pi}{2}$ and the boundary of the domain over which the torus is a bigraph has two connected components which have constant geodesic curvature $\frac{1}{2H}$ in $\Sp^2\times\R$ (with respect to the outer conormal vector field).
\end{itemize}
These two families are represented in figure \ref{fig:S2xR}. We remark that the maximum height of a CMC torus is exactly a half of that of the corresponding sphere for the same mean curvature.

\begin{figure}
\centering\includegraphics[height=3.2cm]{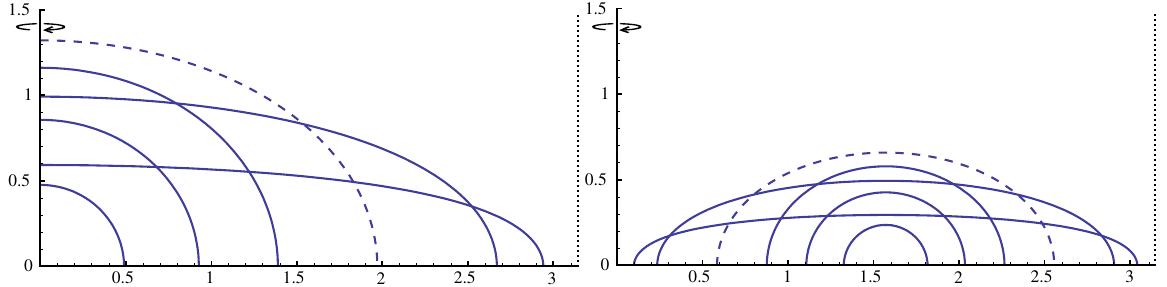}
\caption{On the left, rotationally invariant CMC spheres and, on the right, rotationally invariant CMC tori, in $\Sp^2\times\R$. In both cases, the represented values of $H$ are 0.05, 0.12, 0.331372, 0.6, 1 and 2. The horizontal axis measures the intrinsic distance in $\Sp^2$ while the vertical one is the real line. The maximum height is attained for $H\approx0.331372$, which is drawn as a dashed line.}\label{fig:S2xR}
\end{figure}

\subsection{Invariant surfaces under hyperbolic translations in $\H^2\times\R$}

In this section, we will work in the upper halfplane model $\H^2\times\R=\{(x,y,t)\in\R^3:y>0\}$ endowed with the metric $(dx^2+dy^2)/y^2+dt^2$. Up to conjugation by an ambient isometry, the $1$-parameter group of hyperbolic translations may be considered to be $\{\Phi^h_s\}_{s\in\R}$, where
\[\Phi_s^h:\H^2\times\R\rightarrow\H^2\times\R,\quad\quad\Phi_s^h(x,y,t)=(xe^s,ye^s,t).\]
First of all, we observe that, as the orbit of any point is the horizontal Euclidean straight line which joins the point to a point in the axis $x=y=0$, we can consider the plane $y=1$ as the orbit space of this group of transformations.

Let us take a curve $\gamma(t)=(x(t),1,h(t))$ for some $C^2$ functions $x,h$ defined in some interval of the real line. Thus, a surface invariant by $\{\Phi^h_s\}_{s\in\R}$ can be parametrized as
\begin{equation}\label{eqn:param_hiperbolico}
X(u,t)=\left(x(t)e^u,e^u,h(t)\right).
\end{equation}
It is a straightforward computation to check that the mean curvature of this parametrization is given by
\begin{equation}\label{eqn:H_hiperbolico}
H=\frac{-xh'((1+x^2)(h')^2+2(x')^2)+(1+x^2)(h'x''-x'h'')}{2((1+x^2)(h')^2+(x')^2)^{3/2}}.
\end{equation}
In order to simplify this equation, we will reparametrize the curve $\gamma$ in such a way the denominator simplifies. Observe that we can suppose that $(h')^2+(x')^2/(1+x^2)=1$ so there exists a $C^1$ function $\alpha$ such that $h'=\cos\alpha$ and $x'=\sqrt{1+x^2}\sin\alpha$. Now, we can obtain expressions for $x''$ and $h''$ just by taking derivatives in these identities. If we substitute the results in equation (\ref{eqn:H_hiperbolico}), we get
\[H=\frac{\sqrt{1+x^2}\alpha '-x\cos\alpha}{2 \sqrt{1+x^2}}.\]
The proof of the following lemma is now trivial.

\begin{lema}
The parametrized surface defined in (\ref{eqn:param_hiperbolico}) has constant mean curvature $H\in\R$ if and only if the functions $(x,h,\alpha)$ satisfy the following ODE system
\begin{equation}\label{eqn:sistema_hiperbolico}
\left\{\begin{array}{l}
h'=\cos\alpha\\
x'=\sqrt{1+x^2}\sin\alpha\\
\alpha'=2H+\frac{x\cos\alpha}{\sqrt{1+x^2}}
\end{array}\right.
\end{equation}
Furthermore, the energy function $E=-2Hx-\sqrt{1+x^2}\cos\alpha$ is constant along any solution.
\end{lema}

We will restrict ourselves to the case $H>1/2$. Plugging the expression of the energy into the second equation in (\ref{eqn:sistema_hiperbolico}), it is not difficult to conclude that $x$ verifies the equation
\[
(x')^2=(1-E^2)-4HEx+(1-4H^2)x^2.
\]
As $H>\frac{1}{2}$, the RHS has two different real roots as a polynomial in $x$ and, if we factor it, the equation can be expressed, up to a sign, as
\[\frac{x'}{\sqrt{(4H^2+E^2-1)-\left((4H^2-1)x-2HE\right)^2}}=\frac{\pm 1}{\sqrt{4H^2-1}},\]
from where it is easy to deduce that there exists $c_0\in\R$ such that
\begin{equation}\label{eqn:x_hiperbolico}
x(t)=\frac{2HE}{4H^2-1}+\frac{\sqrt{4 H^2+E^2-1}}{4 H^2-1}\sin
   \left(\pm t\sqrt{4H^2-1}+c_0\right).
\end{equation}
After a translation and a reflection in the parameter $t$, we will suppose without loss of generality that $c_0=0$ and the $\pm$ sign is positive. Now, we can integrate $h$ by taking into account the identity $h'=\cos\alpha=(E+2H)/\sqrt{1+x^2}$, and we get
\begin{equation}\label{eqn:h_hiperbolico}
h(t)=h(0)+\int_0^t\frac{(8 H^2-1)E+2H\sqrt{4H^2+E^2-1} \sin\left(s\sqrt{4H^2-1} \right)}{\sqrt{\left(1-4H^2\right)^2+\left(\sqrt{4H^2+E^2-1} \sin(s\sqrt{4H^2-1})+2HE\right)^2}}\dt s.
\end{equation}
Finally, we are able to characterize the surfaces we were looking for. Some pictures of them are drawn in Figure \ref{fig:H2xR}.

\begin{prop}\label{unicidad_cilindro_H2xR}
Let $(x,h,\alpha)$ be a solution of (\ref{eqn:sistema_hiperbolico}) with energy $E\in\R$ for some $H>\frac{1}{2}$. Then, the generated invariant surface can be extended to an $H$-bigraph if and only if $E=0$. In this case, the generating curve can be reparametrized, up to an ambient isometry, as
\[
\left.\begin{array}{l}
x(r)=\displaystyle\frac{1}{\sqrt{4H^2-1}}\sin r\\
h(r)=\displaystyle\frac{2H}{\sqrt{4 H^2-1}}\arctan\frac{\cos r}{\sqrt{4H^2-1+\sin^2r}}
\end{array}\right\},\quad r\in\R.\]
\end{prop}

\begin{proof}
Let us split $h(t)=h_1(t)+h_2(t)$ in (\ref{eqn:h_hiperbolico}) by splitting the integrand in two additive terms which correspond to the two terms in its numerator. The first term does not vanish unless $E=0$ so $h_1$ is  monotonic and the second one is an odd periodic function in $s$ which vanishes at $s=k\pi/\sqrt{4H^2-1}$ for any $k\in\Z$. On the other hand, if the parametrization interval contains $t=0$, from (\ref{eqn:x_hiperbolico}) we deduce that $|t|\leq\pi/(2\sqrt{4H^2-1})$ so the surface is a graph and, furthermore, the points at which the normal vector field is horizontal must satisfy $x'=0$, so the parametrization interval must be $|t|\leq\pi/(2\sqrt{4H^2-1})$. Now, as the integral of $h_2'$ over $[-\pi/(2\sqrt{4H^2-1}),\pi/(2\sqrt{4H^2-1})]$ vanishes, we have
\[h\left(\frac{\pi}{2\sqrt{4H^2-1}}\right)-h\left(\frac{-\pi}{2\sqrt{4H^2-1}}\right)=h_1\left(\frac{\pi}{2\sqrt{4H^2-1}}\right)-h_1\left(\frac{-\pi}{2\sqrt{4H^2-1}}\right).\]
The RHS term vanishes if and only if $h_1$ identically vanishes as it is monotonic and $h_1$ vanishes if and only if $E=0$. The expressions given in the statement follow from a direct computation in (\ref{eqn:x_hiperbolico}) and (\ref{eqn:h_hiperbolico}) for $E=0$ and from the substitution $r=t\sqrt{4H^2-1}$. Observe that there is no restriction in taking $r\in\R$ because this parametrization generates the whole bigraph.
\end{proof}

In the parametrization given in the statement of Proposition \ref{unicidad_cilindro_H2xR}, observe that the maximum height is attained for $r=0$ and the surface is a bigraph over a domain whose boundary consists of two hypercycles which have constant geodesic curvature in $\H^2$ with respect to the outer conormal vector field, equal to $\frac{-1}{2H}$. Furthermore, the maximum height is exactly a half of that of the corresponding CMC sphere.

\subsection{Invariant surfaces under parabolic translations}
In this case, we will also consider the upper halfplane model for $\H^2$ so, up to conjugation by an ambient isometry, the $1$-parameter group of parabolic translations is $\{\Phi_s^p\}_{s\in\R}$, where 
\[\Phi_s^p:\H^2\times\R\rightarrow\H^2\times\R,\quad\quad\Phi_s^p(x,y,t)=(x+s,y,t).\]
Hence, the orbit of any point in $\H^2$ is a horizontal Euclidean line parallel to the plane $y=0$. Thus, the orbit space may be considered to be the Euclidean plane $x=0$ so the generating curve can be thought as $\gamma(t)=(0,y(t),h(t))$ and a surface invariant by $\{\Phi^p_s\}_{s\in\R}$ can be parametrized as
\[X(u,t)=\left(s,y(t),h(t)\right).\]
It is straightforward to check that the mean curvature of this parametrization is 
\begin{equation}\label{eqn:parabolic_H}
H=-\frac{y^2 \left(-h''y'+h'y''+y(h')^3\right)}{2\left(y^2(h')^2+(y')^2\right)^{3/2}}.
\end{equation}
Furthermore, there is no loss of generality in supposing that the curve $\gamma$ is parametrized by its arc-length, i.e. $1=\|\alpha'\|^2=(y')^2/y^2+(h')^2$. Hence, we can take an auxiliary function $\alpha$, determined by $y'=y\sin\alpha$, $h'=\cos\alpha$. Substituting these equalities in (\ref{eqn:parabolic_H}), it simplifies to the following ODE system
\begin{equation}\label{eqn:sistema_parabolico}
\left\{\begin{array}{l}
y'=y\sin\alpha,\\
h'=\cos\alpha,\\
\alpha'=-2H-\cos\alpha.
\end{array}\right.
\end{equation}
Observe that, if we assume an initial condition $\alpha(0)=\alpha_0\in[0,2\pi]$, the third equation has a unique solution. Let us focus in the case $H>\frac{1}{2}$ which is the most interesting for our purposes and allows us to integrate the function $\alpha$ as
\begin{equation}\label{eqn:alpha_parabolico}
\alpha(t)=2\arctan\left(\frac{(2 H+1)}{\sqrt{4 H^2-1}} \tan\left(\frac{1}{2} \sqrt{4 H^2-1}(t-c_0)\right)\right),
\end{equation}
for some $c_0\in\R$ depending on $\alpha_0$. We emphasize that this formula defines $\alpha:\R\rightarrow\R$ as a strictly increasing diffeomorphism by considering all the branches of the function $\arctan$ and extending it by continuity, so the uniqueness of solution guarantees that every solution is considered in (\ref{eqn:alpha_parabolico}). We will suppose, after a translation in the parameter $t$, that $c_0=0$. By plugging expression (\ref{eqn:alpha_parabolico}) into the first two equations in (\ref{eqn:sistema_parabolico}), we can integrate $h$ and $y$ to obtain
\begin{equation}\label{eqn:h_parabolico}
\begin{array}{l}
y(t)=c_1 \left(\cos \left(t\sqrt{4 H^2-1}\right)+2 H\right),\\
h(t)=\alpha(t)+2Ht+c_2,
\end{array}
\end{equation}
for some constants $c_1>0$ and $c_2\in\R$ which can be supposed to be $c_1=1$ (after a hyperbolic translation) and $c_2=0$ (after a vertical translation). 

\begin{prop}
There are no embedded bigraphs in $\mathbb{H}^2\times\R$ invariant under parabolic translations with constant mean curvature $H>\frac{1}{2}$.
\end{prop}

\begin{proof}
Observe that such a graph must be given by a triple $(y,h,\alpha)$ satisfying (\ref{eqn:sistema_parabolico}) so (\ref{eqn:alpha_parabolico}) and (\ref{eqn:h_parabolico}) are also satisfied. The values of $t\in\R$ for which $y'=0$ are $t_k=k\pi/\sqrt{4H^2+1}$ for any $k\in\Z$ (these ones correspond to the points in the surface whose tangent plane is vertical). Now from (\ref{eqn:alpha_parabolico}) and (\ref{eqn:h_parabolico}) it is easy to check that $h(t_k)\neq h(t_{k+1})$ for every $k\in\Z$, which makes impossible the surface to be a bigraph. 
\end{proof}

\section{Boundary curvature estimates}\label{sec_general_estimate}

\noindent Let us suppose along this section that $\Sigma\subseteq M\times\R$ is a graph over a domain $\Omega\subseteq M$ with constant mean curvature $H>0$. Following the ideas given in \cite{aleg1}, for any given $c\in\R$ with $c+4H^2>0$, we will consider the function $g:[-1,1]\rightarrow\R$ determined by
\begin{equation}\label{eqn:g}
g'(t)=\frac{4H}{4H^2+c(1-t^2)},\quad\quad g(0)=0,
\end{equation}
which is strictly increasing and allows us to define the smooth function $\psi=h+g(\nu)\in C^\infty(\Sigma)$, where $h$ and $\nu$ are the height and angle functions, respectively. In the sequel, $K$ will denote the Gaussian curvature of $\Sigma$, $K_M$ the Gaussian curvature of $M$ extended to $M\times\R$ by making it constant along the vertical geodesics, and $A$ will be the shape operator of $\Sigma$. Note that Gauss equation reads $\det(A)=K-K_M\nu^2$.

As we are interested in applying the boundary maximum principle for the laplacian to $\psi$, we will need to work out $\Delta\psi$ (where the laplacian is computed in the surface $\Sigma$) and $\frac{\partial\psi}{\partial\eta}$, where $\eta$ is some outer conormal vector field to $\partial\Sigma$.  Next lemma will be useful.

\begin{lema}\label{lema_formulas}
In the previous situation, the following equalities hold.
\begin{itemize}
 \item[i)] $\nabla h=E_3^\top$,
 \item[ii)] $\Delta h=2H\nu$,
 \item[iii)] $\nabla\nu=-AE_3^\top$,
 \item[iv)] $\Delta\nu=\left(2K-4H^2-K_M(1+\nu^2)\right)\nu$.
\end{itemize}
\end{lema}

\begin{proof}
The identities for the gradient and the laplacian of $h$ are easy to check as $h$ is the restriction to $\Sigma$ of the height function in $M\times\R$ (see also \cite[Lemma 3.1]{rose2}). On the other hand, the gradient of $\nu=\langle N,E_3\rangle$ satisfies
\[\langle\nabla\nu,X\rangle=X(\langle N,E_3\rangle)=\langle\nabla_XN,E_3\rangle=\langle-AX,E_3\rangle=\langle X,-AE_3^\top\rangle\]
for any vector field $X$ on $\Sigma$, so $\nabla\nu=-AE_3^\top$. Finally, since the vertical translations are isometries of $M\times\R$, $\nu$ is a Jacobi function, i.e. $\nu$ lies in the kernel of the linearized mean curvature operator $L=\Delta+|A|^2+\Ric(N)$ on $\Sigma$ so we can compute its laplacian from $L\nu=0$ and obtain
\[\Delta\nu=-\left(|A|^2+\Ric(N)\right)\nu=\left(2K-4H^2-K_M(1+\nu^2)\right)\nu,\]
where we have used the Gauss equation and the well-known identities $|A|^2=4H^2-2\det(A)$ and $\Ric(N)=K_M(1-\nu^2)$.
\end{proof}

On the other hand, we need to obtain some suitable expression for the modulus of the Abresch-Rosenberg differential, in the case $M=\M^2(c)$. If we take a conformal parametrization $(U,z)$ in $\Sigma$, this quadratic differential can be written as
\[Q=(2Hp-ch_z^2)\dt z^2\]
(see \cite{fm2}), where $p\dt z^2=\langle-\nabla_{\p_z}N,\p_z\rangle\dt z^2$ is the Hopf differential and $h_z=\frac{\p h}{\p z}$. Although this expression depends on the parametrization, we may consider the function 
\begin{align}
q=\frac{4}{\lambda^2}|Q|^2&=\tfrac{4}{\lambda^2}\left(4H^2|p|^2+c^2|h_z|^4-2cH(ph_{\bar z}^2+\bar p h_z^2)\right)\notag\\
&=4H^2(H^2-\det(A))+\tfrac{c^2}{4}(1-\nu^2)^2-c(\|\nabla\nu\|^2-(2H^2-\det(A))(1-\nu^2)),\label{eqn:q}
\end{align}
where $\lambda$ is the conformal factor of the induced metric in $\Sigma$. Then, $q$ is well-defined and smooth on the whole $\Sigma$.

\begin{obse}
If $\Sigma$ is a constant mean curvature surface in $\M^2(c)\times\R$ whose Abresch-Rosenberg differential identically vanishes, then $\Sigma$ is (locally) invariant by a 1-parameter subgroup of isometries of $\M^2(c)\times\R$ which acts trivially on the vertical lines (see \cite[Lemma 6.1]{esr3}). Moreover, if $M$ is a Riemannian quotient of $\M^2(c)$ and $\Sigma\subset M\times\R$ is such that $Q=0$, then the lifted surface $\wt\Sigma\subset\M^2(c)\times\R$ also satisfies $Q=0$.
\end{obse}

Back to the computation of $\Delta\psi$ and taking into account the formulas in Lemma \ref{lema_formulas} and identity (\ref{eqn:q}), we get
\begin{equation}\label{eqn:laplaciano}
\Delta\psi=\Delta h+g'(\nu)\Delta\nu+g''(\nu)\|\nabla\nu\|^2=\frac{-8Hq\nu}{\left(4 H^2+c(1-\nu^2)\right)
^2}-\frac{4H\nu\left(1-\nu^2\right) (K_M-c)}{4 H^2+c(1-\nu^2)}.
\end{equation}

Finally, we are interested in working out $\frac{\p\psi}{\p\eta}$ along $\partial\Sigma$, where we have considered the outer conormal vector field to $\partial\Sigma$ in $\Sigma$ given by $\eta=-E_3^\top$ (it does not matter which outer conormal vector field is chosen as the only needed information is the sign of $\frac{\partial\psi}{\partial\eta}$). Hence, 
\begin{align*}
\frac{\p h}{\p\eta}&=\langle\nabla h,\eta\rangle=\langle E_3^\top,-E_3^\top\rangle=-\|E_3^\top\|^2,\\
\frac{\p\nu}{\p\eta}&=\langle\nabla\nu,\eta\rangle=\langle -AE_3^\top,-E_3^\top\rangle=\langle\overline\nabla_{E_3^\top}E_3^\top,N\rangle.
\end{align*}
However, if we parametrize $\partial\Sigma$ by $\gamma$ with $\|\gamma'\|=1$, then $\{E_3^\top/\|E_3^\top\|,\gamma'\}$ is an orthonormal basis of $T\Sigma$, and it is clear that
\[2H=\left\langle\frac{1}{\|E_3^\top\|^2}\overline\nabla_{E_3^\top}E_3^\top+\overline\nabla_{\gamma'}\gamma',N\right\rangle\]
so $\frac{\p\nu}{\p\eta}=2H\|E_3^\top\|^2+\|E_3^\top\|^3\kappa_g$. Here, $\kappa_g$ denotes the geodesic curvature of $\partial\Omega
=\partial\Sigma$ in the base $M$ with respect to $-\|E_3^\top\|^{-1}(N-\nu E_3)$, the outer conormal vector field to $\partial\Omega$ in $M$. Finally, we obtain
\begin{equation}\label{eqn:kg}
\frac{\p\psi}{\p\eta}=\frac{\p h}{\p\eta}+g'(\nu)\frac{\p\nu}{\p\eta}=\|E_3^\top\|^2\left(-1+g'(\nu)(2H+\|E_3^\top\|\kappa_g)\right).
\end{equation}

Note that any outer conormal vector field to $\partial\Omega$ in $M$ is a linear combination of $N$ and $E_3$, which is the key property to relate the geometries of $\Sigma$ and $M$.

Moreover, these computations allow us to give an optimal bound for the geodesic curvature of the boundary of the domain of a compact $H$-graph with a capillarity boundary condition. Theorem~\ref{teor:acotacion_kg} deals with the case the angle along the boundary is identically zero (see Remark~\ref{rmk:acotacion_kg_general} for the general case,).

In order to state the theorem in a more convenient way, observe that if $\Sigma\subseteq M\times\R$ is a compact embedded $H$-surface, then $\Sigma$ is a symmetric bigraph with respect to some slice $M\times\{t_0\}$  by applying the Alexandrov reflection principle~\cite{Alex} to vertical reflections (this slice  may be supposed to be $M\times\{0\}$after a vertical translation). Now it is obvious that $\Sigma$ intersects orthogonally such a slice, and it splits $\Sigma$ in two parts which the computations in this section can be applied to.

\begin{teor}\label{teor:acotacion_kg}
Let $\Sigma\subseteq M\times\R$ be a compact embedded $H$-surface with $H>0$. Then, it is an $H$-bigraph over a compact regular domain $\Omega\subseteq M$ with  zero values in $\partial\Omega$ after a vertical translation. Let us suppose that $c=\inf\{K_M(p):p\in\Omega\}$ satisfies $4H^2+c>0$. Then,
\begin{equation}\label{eqn:geodesic-general}
\kappa_g\geq -4H+\frac{c}{4H},
\end{equation}
where $\kappa_g$ is the geodesic curvature of $\partial\Omega$ in $M$ with respect to the outer conormal vector field.

Furthermore, if there exists $p\in\partial\Omega$ such that equality holds in \eqref{eqn:geodesic-general}, then $\Omega$ has constant curvature and $\Sigma$ has zero Abresch-Rosenberg differential.
\end{teor}

\begin{proof}
Let us consider the function $\psi=h+g(\nu)$, defined in terms of (\ref{eqn:g}). Since $\nu\leq 0$ and $K_M\geq c$ in $\Sigma$, equation \eqref{eqn:laplaciano} insures that $\Delta\psi\geq0$ in $\Sigma$. As $\Omega$ is compact and $\psi$ is constant along $\p\Omega$, the maximun principle in the boundary guarantees that $\frac{\partial\psi}{\partial\eta}\geq0$. By using equation \eqref{eqn:kg}, this inequality is equivalent to \eqref{eqn:geodesic-general}. Equality holds at some point of $\p\Omega$ if and only if  $\psi$ is constant so, from \eqref{eqn:laplaciano}, we get that $q=0$ and $K_M=c$ in $\Omega$.
\end{proof}

\begin{obse}\label{rmk:acotacion_kg_general}
The same argument in the proof of Theorem~\ref{teor:acotacion_kg} works when we assume $\Sigma$ is an $H$-graph over a regular domain with $h=0$ and $\nu=\nu_0$ in $\partial\Omega$ for some $-1<\nu_0\leq 0$. In this case, the lower bound for the geodesic curvature becomes
\begin{equation}\label{eqn:geodesic-general-arbitrario}
\kappa_g\geq\frac{-4H^2+c(1-\nu_0^2)}{4H\sqrt{1-\nu_0^2}},
\end{equation}
which coincides with~\eqref{eqn:geodesic-general} for $\nu_0=0$. If equality holds at some point $p\in\partial\Omega$, then $\Omega$ has constant curvature and $\Sigma$ has zero Abresch-Rosenberg differential.

Observe that if we suppose $-1<\nu\leq\nu_0$ in $\partial\Omega$ for some $\nu_0\leq 0$ (rather than $\nu=\nu_0$ in $\partial\Omega$) and there exists a point $p\in\partial\Omega$ such that $\nu(p)=\nu_0$ and at which~\eqref{eqn:geodesic-general-arbitrario} becomes and equality, then $\Omega$ has constant curvature and $Q=0$ in $\Sigma$. In the case $M=\M^2(c)$, equality holds if and only if $\Sigma$ is a spherical cap of a standard rotational sphere.
\end{obse}

We now adjust the value of $H$ for which the lower bound is exactly zero, which provides a characterization of the rotationally invariant spheres in $\Sp^2(c)\times\R$.

\begin{coro}\label{coro_convexo}
Let $M$ be a orientable complete Riemannian surface with $K_M\geq c>0$ in $M$. Then, each compact embedded $H$-surface in $M\times\R$ with $0<H<\frac{1}{2}\sqrt c$ is an  $H$-bigraph over a connected domain $\Omega$ and $M\smallsetminus\Omega$ is a finite union of curves of non-negative geodesic curvature. Furthermore, either
\begin{itemize}
 \item their geodesic curvature is strictly positive, or
 \item $M=\Sp^2(c)$, $\Omega$ is a closed hemisphere and $\Sigma$ is a rotationally invariant $H$-sphere for $H=\frac{1}{2}\sqrt c$ (in this case $\kappa_g$ identically vanishes).
\end{itemize}
\end{coro}

Observe that, since $\Omega$ is a connected domain with regular boundary and $M$ is topologically a $2$-sphere under the assumptions of the Corollary, the connected components of $M\smallsetminus\Omega$ are topological disks. In the case $M=\Sp^2(c)$, each of these disks is geodesically convex (i.e. every minimizing geodesic joining a pair of points of its boundary is interior to it) since it is bounded by a curve whose geodesic curvature does not change sign. In particular, such a disk must lie in a hemisphere of $\Sp^2(c)$.

\section{Further boundary curvature estimates}

\noindent In this section, we will obtain better estimates for the geodesic curvature of the boundary by assuming restrictions on the maximum height that the surface can reach. In order to achieve this, we will use a technique which has its origins in a paper by Payne and Philippin \cite{paph1} and which has also been used by Ros and Rosenberg in \cite{ror1}.

Let $\Sigma\subseteq M\times\R$ be a constant mean curvature $H>0$ surface which is a graph over a domain $\Omega\subseteq M$ and extends continuously to the boundary of $\Omega$ with zero values. For any given $m>0$, let consider the function $g_m:[-1,1]\rightarrow\R$ determined by
\begin{equation}\label{eqn:g-general}
g_m'(t)=\frac{4mH}{4H^2+c(1-t^2)},\quad\quad g(0)=0,
\end{equation}
which is strictly increasing and satisfies $m\cdot\alpha(c,H,\nu_0)=g_m(\nu_0)-g_m(-1)$. Moreover,
\[X=\frac{2H\nu(2m-1)}{m(1-\nu^2)}E_3^\top-\frac{2H\nu g_m'(\nu)}{m(1-\nu^2)}AE_3^\top\]
is a smooth vector field on $\Sigma\smallsetminus V$, where $V=\{p\in\Sigma:\nu(p)=-1\}$ is the subset of $\Sigma$ with vertical Gauss map. We will consider the second order elliptic operator $L$ on $C^\infty(\Sigma\smallsetminus V)$ given by $Lf=\Delta f+X(f)$, and the function $\psi_m=h+g_m(\nu)\in C^\infty(\Sigma)$. We are now interested in working out $L\psi_m$. By using Lemma \ref{lema_formulas} and the identity $A^2=2HA-\det(A)\cdot\mathrm{id}$ in $T\Sigma$, we obtain
\begin{equation}\label{eqn:operador-psi}
L\psi_m=-\frac{4H(m-1)(m-\frac{1}{2})\nu}{m}-\frac{4Hm\nu(1-\nu^2)(K_M-c)}{4H^2+c(1-\nu^2)}.
\end{equation}
The second term in the RHS is non-negative since $K_M\geq c$. Moreover, for $m\geq 1$ or $m\leq \frac{1}{2}$, the first term is also positive so the function $\psi_m$ verifies $L\psi_m\geq0$ in $\Sigma\smallsetminus V$. Thus, it is possible to apply the maximum principle for the operator $L$ in $\Sigma\sm V$, which insures that $\psi_m$ cannot achieve an interior maximum in $\Sigma\sm V$ unless it is constant. 

\begin{lema}\label{lema_funcion_constante}
Let $\Sigma\subseteq M\times\R$ be a constant mean curvature $H>0$ graph over a (not necessarily compact) domain $\Omega\subseteq M$ and suppose that $c=\inf\{K_M(p):p\in\Omega\}>-4H^2$. If $\psi_m$ is constant in $\Sigma$ for some $m\leq\frac{1}{2}$, then
\begin{itemize}
 \item[a)] $m=\frac{1}{2}$ and $K_M$ is constant in $\Omega$,
 \item[b)] $\Sigma$ is invariant by a $1$-parameter group of isometries which preserve the height function.
\end{itemize}
In particular, if $c>0$ and $M=\Sp^2(c)$, then $\Sigma$ is a compact rotationally invariant torus and, if $c\leq 0$ and $M=\H^2(c)$, then $\Sigma$ is an invariant horizontal cylinder, both described in Section \ref{sec_ejemplos}.
\end{lema}

\begin{proof}
If $\psi_m$ is constant, then $L\psi_m=0$ so from (\ref{eqn:operador-psi}) we get that $(m-1)(m-\frac{1}{2})\leq 0$ which is only possible if $m=\frac{1}{2}$. Then, as equality in (\ref{eqn:operador-psi}) holds, $K_M$ must be constant in $\Sigma\sm V$ so it is constant in $\Sigma$ as $K_M$ is continuous and $V$ has empty interior.

Now, suppose that $m=\frac{1}{2}$ and $\psi_{1/2}$ is constant. On one hand, from $\nabla\psi_{1/2}=0$ we obtain $AE_3^\top=\frac{1}{g_{1/2}'(\nu)}E_3^\top$ so $E_3^\top$ must be a principal direction and $\frac{1}{g_{1/2}'(\nu)}$ its corresponding principal curvature. We also deduce the following expressions: 
\begin{align}
\det(A)&=\frac{1}{g_{1/2}'(\nu)}\left(2H-\frac{1}{g_{1/2}'(\nu)}\right),&
\|\nabla\nu\|^2&=\langle AE_3^\top,AE_3^\top\rangle=\frac{1-\nu^2}{g_{1/2}'(\nu)^2}.\label{eqn:nabla}
\end{align}
If we consider the differentiable function $f:\ ]-1,1[\ \rightarrow\R$ determined by
\[f'(t)=\frac{1}{\sqrt{(1-t^2)(4H^2+c(1-t^2))}},\quad\quad f(0)=0,\]
and take into account (\ref{eqn:nabla}) and Lemma \ref{lema_formulas}, it is easy to check that
\begin{align*}
\Delta(f(\nu))&=f''(\nu)\|\nabla\nu\|^2+f'(\nu)\Delta\nu\\
&=\frac{f''(\nu)}{g_{1/2}'(\nu)^2}(1-\nu^ 2)+f'(\nu)(2\det(A)-4H^2-c(1-\nu^2))\nu=0,
\end{align*}
where we have also used that $K_M$ is constant by item (a). As $f(\nu)$ is a non-constant harmonic function on $\Sigma\sm V$, we can (at least locally) take a conformal parameter $z=x+iy$ on $\Sigma\smallsetminus V$ with $x=f(\nu)$. Now, by repeating the arguments given in \cite[Lemma 6.1]{esr3}, we conclude that $\Sigma$ is invariant by a $1$-parameter group of isometries of $\M^2(c)\times\R$.

Since all the fundamental data only depend on $x$, that group can be understood as translations in the $y$-direction. Thus, if we show that $h_y$, the derivative of the height function with respect to $y$, vanishes, item (b) will be checked out. As the parameter is conformal, $h_y=\langle\nabla h,\partial_y\rangle=-\langle JE_3^\top,\partial_x\rangle$, but $\partial_x$ has the same direction as $\nabla\nu=-AE_3^\top$ because $x=f(\nu)$ and $AE_3^\top=\frac{1}{g_{1/2}'(\nu)}E_3^\top$ from $\nabla\psi_{1/2}=0$, so $h_y=0$.

Finally, just by checking all the invariant surfaces under $1$-parameter groups of isometries preserving the height function (cf. Section \ref{sec_ejemplos}), it is easy to realize that those mentioned in the statement are the only ones which satisfy that $\psi_{1/2}$ is constant.
\end{proof}

\begin{teor}\label{thm:acotacion-kg-adicional}
Let $\Sigma\subseteq M\times\R$ be a $H$-bigraph over a compact regular domain $\Omega\subset M$ and suppose that $c=\inf\{K_M(p):p\in\Omega\}$ satisfies that $4H^2+c>0$. If there exists $0<m\leq\frac{1}{2}$ such that $|h|\leq m\cdot\alpha(c,H,0)$, then the following lower bound for the geodesic curvature of $\partial\Omega$ in $M$ (with respect to the outer conormal vector field) holds:
\[\kappa_g\geq\frac{(4-8m)H^2+c}{4mH}.\]
\end{teor}

\begin{proof}
We will suppose that $\nu=\nu_0$ for some $-1<\nu_0\leq0$ along $\partial\Sigma$, so the Theorem will follow from making $\nu_0=0$ (see also Remark~\ref{rmk:acotacion-kg-adicional} below).
Let us consider the function $\psi_m=h+g_m(\nu)\in C^\infty(\Sigma)$, which verifies $L\psi_m\geq 0$ in view of (\ref{eqn:operador-psi}). As $\Sigma$ is compact, there exists a point $p_0\in\Sigma$ where $\psi_m$ attains its maximum. We distinguish three possibilities:
\begin{itemize}
 \item If $p_0$ is an interior point of $\Sigma\sm V$, then $\psi_m$ is constant in $\Sigma$, which implies that the maximum is also attained in $\p\Sigma$.
 \item If $p_0\in\partial\Sigma$, then such a maximum is attained in the whole boundary $\p\Sigma$ since $(\psi_m)_{|\partial\Sigma}$ is constant. Then the boundary maximum principle for the operator $L$ guarantees that $\frac{\p\psi_m}{\p\eta}\geq 0$ along $\partial\Sigma$. It is straightforward to check from (\ref{eqn:kg}) that this is equivalent to the inequality in the statement above.
 \item If $p_0\in V$, then $\nu(p_0)=-1$. Observe that $h\leq m\cdot\alpha(c,H,\nu_0)=g_m(\nu_0)-g_m(-1)$, so $\psi_m\leq\psi_m(p_0)=h(p_0)+g_m(-1)=h(p_0)-m\cdot\alpha(c,H,\nu_0)+g_m(\nu_0)\leq g_m(\nu_0)$ and, since $\psi_m$ is equal to $g_m(\nu_0)$ in $\partial\Sigma$, the maximum is also attained in the boundary, which reduces this case to the previous one.\qedhere
\end{itemize}
\end{proof}

We now adjust the constant $0<m\leq\frac{1}{2}$ to guarantee the convexity of the boundary, as we did in Corollary \ref{coro_convexo}.

\begin{coro}\label{coro:convexidad-adicional}
Let $\Sigma\subseteq M\times\R$ be an $H$-bigraph over a compact regular domain $\Omega$ with $H>0$. Suppose that $c=\inf\{K_M(p):p\in\Omega\}$ satisfies that $4H^2+c>0$. In any of the situations:
\begin{itemize}
 \item[i)] $c\geq 0$ and $h\leq\frac{1}{2}\alpha(c,H,0)$ in $\Sigma$, or
 \item[ii)] $c<0$ and $h\leq \frac{4H^2+c}{8H^2}\alpha(c,H,0)$ in $\Sigma$,
\end{itemize}
the boundary $\p\Omega$ is convex in $M$ with respect to the outer conormal vector field.
\end{coro}

\begin{obse}\label{rmk:acotacion-kg-adicional}
The proof of Theorem~\ref{thm:acotacion-kg-adicional} is also valid when $\Sigma$ is an $H$-graph, $H>0$, over a compact regular domain $\Omega\subset M$ with $h=0$ and $\nu=\nu_0$ in $\partial\Omega$ for some  $-1<\nu_0\leq0$. In this case, if we suppose that $|h|\leq m\cdot\alpha(c,H,\nu_0)$, then the lower bound for the geodesic curvature can be improved to the following one:
\[\kappa_g\geq\frac{(4-8m)H^2+c(1-\nu_0^2)}{4mH\sqrt{1-\nu_0^2}}.\]
Moreover, the situations in which we can guarantee that $\partial\Omega$ is convex with respect to the outer conormal vector field (as in Corollary~\ref{coro:convexidad-adicional}) become the following ones under these new capillarity assumptions:
\begin{itemize}
 \item[i)] $c\geq 0$ and $h\leq\frac{1}{2}\alpha(c,H,\nu_0)$ in $\Sigma$, or
 \item[ii)] $c<0$ and $h\leq \frac{4H^2+c(1-\nu_0^2)}{8H^2}\alpha(c,H,\nu_0)$ in $\Sigma$,
\end{itemize}
\end{obse}

We finally wonder whether the compactness hypothesis for the domain of the graph can be removed. In order to achieve this, we will restrict ourselves to the homogeneous ambient space $\M^2(c)\times\R$ and $\nu_0=0$ (that is, $\Sigma$ is an $H$-bigraph), where the technique developed by Ros and Rosenberg in \cite{ror1} can be adapted.

\begin{teor}\label{teor_homogeneos}
Let $\Sigma\subseteq\M^2(c)\times\R$ be a properly embedded $H$-bigraph over a domain $\Omega\subseteq\M^2(c)$ with $4H^2+c>0$, symmetric with respect to $\M^2(c)\times\{0\}$, and suppose that there exists $0<m\leq\frac{1}{2}$ such that $|h|\leq m\cdot\alpha(c,H,0)$ in $\Sigma$. Then, the following lower bound for the geodesic curvature of $\partial\Omega$ in $\M^2(c)$ (with respect to the outer conormal vector field) holds:
\[\kappa_g\geq\frac{(4-8m)H^2+c}{4mH}.\]
Furthermore, if there exists a point in $\partial\Omega$ where the equality is attained, then:
\begin{itemize}
 \item[i)] $\Sigma$ is a rotationally invariant torus (see Section \ref{sec_ejemplos_rotacionales}) if $c>0$,
 \item[ii)] $\Sigma$ is an invariant cylinder under horizontal translations if $c=0$, and
 \item[iii)] $\Sigma$ is an invariant cylinder under hyperbolic translations (see Section \ref{unicidad_cilindro_H2xR}) if $c<0$.
\end{itemize}
\end{teor}

\begin{proof}
Let us consider the same function $\psi_m\in C^\infty(\Sigma)$ as before. If $\psi_m$ attained its maximum or $\sup_\Sigma\psi_m\leq0$, we could reason in the same way we did for the compact case and the proof would be finished. Otherwise, let us take a sequence $\{p_n\}\subseteq\Sigma^+=\{p\in\Sigma:h(p)>0\}$ such that $\{\psi_m(p_n)\}$ converges to $\sup\psi_m$, and distinguish two cases. 
\begin{itemize}
 \item If $\lim\{h(p_n)\}=0$, then $\psi_m(p_n)=h(p_n)+g_m(\nu(p_n))\leq h(p_n)\rightarrow 0$ from where $\sup_\Sigma\psi_m\leq 0$ and we are done.
 \item If $\{h(p_n)\}$ does not converge to zero, we can suppose that $\{h(p_n)\}\rightarrow a>0$ and $h(p_n)>\frac{a}{2}$ for all $n\in\mathbb{N}$ without loss of generality. Given $0<\epsilon<\frac{a}{2}$, as $\Sigma^+$ is stable and the distance from $p_n$ to $\p\Sigma^+$ is bounded away from zero, the surface $\Sigma_\epsilon=\{p\in\Sigma:h(p)>\epsilon\}$ has bounded second fundamental form. Ambient isometries allow us to translate $\Sigma_\epsilon$ horizontally so that $p_n$ is over some fixed point $q_0\in\M^2(c)$ and standard convergence arguments make possible to consider $\Sigma_\infty(\epsilon)$, the limit $H$-graph of a subsequence of these translated surfaces. The corresponding function in $\Sigma_\infty(\epsilon)$, given by $\psi_{m,\infty}=h_\infty+g_m(\nu_\infty)\in C^\infty(\Sigma_\infty(\epsilon))$, attains its maximum at the interior point $p_0=(q_0,a)\in\Sigma_\infty(\epsilon)$ (observe that there is convergence in the $C^m$ topology on compact subsets for every $m\in\mathbb{N}$). As $\psi_{m,\infty}$ is subharmonic on $\Sigma_\infty(\epsilon)$, we have that it is constant because of the maximum principle, and Lemma~\ref{lema_funcion_constante} implies that $\Sigma_\infty(\epsilon)$ can be extended to the upper half of one of the bigraphs listed in the statement of the theorem, up to a vertical translation, which will be denoted by $\widetilde\Sigma_\infty$. Note that $\widetilde\Sigma_\infty$ and can be supposed independent of $\epsilon$ by standard diagonal arguments (because decreasing $\epsilon$ just increases the size of the surfaces involved in the limit process). In this situation, $\sup_{\Sigma}\psi_m$ is the height in $\M^2(c)\times\R$ of a point $\widetilde p\in\partial\widetilde\Sigma_\infty$ satisfying $\nu_\infty(\widetilde p)=0$.

Now we prove that $h(\widetilde p)\leq0$ (which finishes the proof since $h(\widetilde p)=\sup_\Sigma\psi_m$). Arguing by contradiction, if $h(\widetilde p)>0$ we could consider $0<\epsilon<\frac{1}{2}h_\infty(\widetilde p)$ and, as the extended limit surface $\widetilde\Sigma_\infty$ does not depend on $\epsilon$, we would be able to find a subsequence of the translated surfaces of $\Sigma_\epsilon$ converging to the extended graph $\widetilde\Sigma_\infty$. Thus, $\widetilde\Sigma_\infty$ contains points at height as close to $\epsilon$ as desired, contradicting the fact that no point in $\widetilde\Sigma_\infty$ has lower height than $\widetilde p$ and $\epsilon<\frac{1}{2}h_\infty(\widetilde p)$.\qedhere
\end{itemize}
\end{proof}

\begin{obse}
Observe that, if the maximum heights of a sequence $\{\Sigma_n\}$ of such $H$-bigraphs tend to zero, then Theorem \ref{teor_homogeneos} insures that the geodesic curvatures of the boundaries diverge uniformly, in the sense that the bound only depends on that maximum height. Thus, the sequence of domains $\Omega_n\subseteq\M^2(c)$ over which $\Sigma_n$ is a bigraph cannot eventually omit any set in $\M^2(c)$ with non-empty interior.
\end{obse}

\section{Intrinsic length estimates}\label{sec:length-estimates}

\noindent Let $\Sigma\subseteq M\times\R$ be an $H$-graph over a compact domain $\Omega\subseteq M$ which extends continuously to the boundary with zero values. Suppose that $K_M\geq c>-4H^2$ in $\Sigma$ for some $c>0$. In Section \ref{sec_general_estimate} we proved that $\psi=h+g(\nu)$ is subharmonic in $\Sigma$, where $g$ is defined in (\ref{eqn:g}) so, if we suppose that $\nu\leq\nu_0$ along $\partial\Sigma$, then $h+g(\nu)\leq g(\nu_0)$, as a consequence of that $g$ is strictly increasing and $h$ vanishes on $\partial\Sigma$. Therefore, as $g$ is also an odd function, we derive that $g(-\nu)\geq h-g(\nu_0)$. Now we can invert the function $g$ and square both sides to obtain
\begin{equation}\label{eqn:angulo_maximo}
\nu^2\geq\zeta(h,\nu_0):=\begin{cases}
\frac{c+4H^2}{c}\tanh^2\left(\frac{\sqrt{c^2+4H^2c}}{4H}(h-g(\nu_0))\right)&\text{if }c<0,\\
H^2(h-g(\nu_0))^2&\text{if }c=0,\\
\frac{c+4H^2}{-c}\tan^2\left(\frac{\sqrt{-c^2-4H^2c}}{4H}(h-g(\nu_0))\right)&\text{if }c>0.\\
\end{cases}
\end{equation}

Let $\gamma:[a,b]\rightarrow\Sigma$ be a smooth curve which is parametrized by arc-length and let $\eta$ be a smooth unit vector field along $\gamma$, orthogonal to $\gamma'$ and $N$. Then, as $\{N,\gamma',\eta\}$ is an orthonormal frame, we have
\[E_3=\langle N,E_3\rangle E_3+\langle\gamma',E_3\rangle\gamma'+\langle\eta,E_3\rangle\eta,\]
and, since $\langle N,E_3\rangle=\nu$ and $\langle\gamma',E_3\rangle=h'(\gamma)$, we deduce that $1=\nu^2+h'(\gamma)^2+\langle\eta,E_3^\top\rangle^2$. Taking into account that $\langle\eta,E_3^\top\rangle^2\geq0$, we finally get $|h'|\leq\sqrt{1-\nu^2}$. Thus, plugging (\ref{eqn:angulo_maximo}) into this inequality, we have
\begin{equation}\label{acotacion_longitud}
\Long(\gamma)\geq\int_0^a\frac{|h'|}{\sqrt{1-\nu^2}}\dt t\geq\int_0^a\frac{-h'}{\sqrt{1-\zeta(h,\nu_0)}}\dt t=\int_{h(a)}^{h(0)}\frac{ds}{\sqrt{1-\zeta(s,\nu_0)}}.
\end{equation}
Considering all the curves that join a point $p$ with the boundary (along which the height vanishes), we obtain the following result:

\begin{teor}\label{length_estimate}
Let $\Sigma\subseteq M\times\R$ be an $H$-graph, $H>0$, over a compact domain $\Omega\subseteq M$ which extends continuously to the boundary with zero values and suppose that $c=\inf\{K_M(p):p\in\Omega\}>-4H^2$. If $\nu\leq\nu_0$ in $\partial\Omega$ for some $-1<\nu_0\leq 0$, then
\[\dist(p,\partial\Sigma)\geq\int_0^{h(p)}\frac{ds}{\sqrt{1-\zeta(s,\nu_0)}}.\]
Furthermore, if there exists $p\in\Sigma$ such that equality holds, then $\Omega$ has constant curvature and $\Sigma$ is a spherical cap of a rotationally invariant sphere.
\end{teor}

Theorem \ref{length_estimate} is a comparison result which may be understood in the following way: take $\Sigma$ in the conditions of the statement and $S\subset\M^2(c)\times\R^+$ a rotationally invariant spherical cap with the same mean curvature $H$ and making a constant angle $\nu_0$ with the slice $\M^2(c)\times\{0\}$. Then, for any $p\in\Sigma$, the distance $d(p,\partial\Sigma)$ is at least $d(q,\partial S)$, where $q$ is any point in $S$ satisfying $h(q)=h(p)$.

\end{document}